\documentclass[twoside]{article}
\usepackage{pstricks, pstricks-add}
\usepackage{pgfplots}
\usepackage{amsmath,amsthm, amsfonts}
\usepackage{amsmath,amsthm, amsfonts, amssymb}
\usepackage{enumerate}
\usepackage{geometry}
\usepackage{secdot}
\usepackage{sectsty}
\allsectionsfont{\normalfont}
\numberwithin{equation}{section}

\usepackage{caption}
\theoremstyle{plain}
\newtheorem{thm}{Theorem}[section]
\newtheorem{prop}[thm]{Proposition}
\newtheorem{cor}[thm]{Corollary}
\newtheorem{lem}[thm]{Lemma}
\theoremstyle{definition}
\newtheorem{defn}[thm]{Definition}
\newtheorem{exa}[thm]{Example}

\newtheorem*{question*}{Question}
\newtheorem{conj}[thm]{Conjecture}
\newtheorem{rem}[thm]{Remark}

\usepackage{enumerate}
\usepackage[utf8]{inputenc}
\usepackage[english]{babel}
 \usepackage{fancyhdr}
 \usepackage{pgf,tikz}
 \usepackage{pgfplots}
  \usepackage{pstricks-add}
\usepackage{multicol}
 \usepackage{enumerate}
 \usepackage{xcolor}
 \usepackage{geometry}
\usepackage{mathrsfs}
\usetikzlibrary{arrows}
\usepackage{mathtools}
\usepackage{wasysym}
\usepackage{float}
\usepackage{array}

\usepackage[linktocpage, bookmarks]{hyperref}

\usepackage{fancyhdr}

\author{First Author \and Second Author}
\title{On the Cycle Structure of Finite Racks and Quandles}

\makeatletter
\newcommand\Author{NAQEEB UR REHMAN}
\let\Title\@title
\makeatother

\begin{document}
\title{\textsc{On the Cycle Structure of Finite Racks and Quandles}}

\author{NAQEEB UR REHMAN}
\date{}
\maketitle



\paragraph{ABSTRACT.}We provide obstructions on the cycle structure of inner automorphisms of finite indecomposable racks and quandles and verify some cases of a conjecture by C. Hayashi.

\begin{center}
	\section{\textsc{Introduction}}
\end{center}
\paragraph{} Racks and quandles are self-distributive algebraic structures whose binary operations are like the conjugation in a group. Finite racks and quandles have been studied in \cite{10} by their inner automorphisms. Based on this perspective, C. Hayashi studied finite indecomposable quandles in \cite{7} and conjectured that any cycle length of an inner automorphism of a finite indecomposable quandle divides the largest cycle length of that inner automorphism. In this paper we will discuss to which extent Hayashi's conjecture is true by providing obstructions on the cycle structure of inner automorphisms of finite indecomposable racks and quandles. The paper is organized as follows. In Section \ref{2.} we review the basic definitions of racks, quandles and their cycle structure. In Section \ref{3.} we provide obstructions on the cycle structure of finite indecomposable racks and quandles. Finally, in Section \ref{4.} we verify some cases of Hayashi's conjecture on the cycle structure of finite indecomposable racks and quandles.

\begin{center}
	\section{\textsc{Racks, Quandles and their Cycle Structure}} \label{2.}
\end{center}

\paragraph{} In this section we recall the basics of racks, quandles and their cycle structure from \cite{1}, \cite{7}, \cite{10}.

\begin{defn}
A \emph{rack} is a pair $(X, \rhd)$, where $X$ is a non-empty set and
$\rhd:X\times X \longrightarrow X$ is a binary operation such that\\
\\[1pt] ${}$ \; (R1) the map $\varphi_x : X \longrightarrow X$, defined by $\varphi_x(y)=x\rhd y$, is bijective for all $x\in X$, and
\\[1pt] ${}$ \; (R2) $x \rhd (y \rhd z) = (x \rhd y) \rhd (x \rhd z)$ for all $x, y, z \in X$ (i.e. $\rhd$ is self-distributive).\\
\\[1pt] ${}$ \; A rack $(X, \rhd)$, or shortly $X$, is called \emph{quandle} if $x \rhd x = x$ for all $x \in X$. A \emph{crossed set} is a quandle $X$ such that $x \rhd y = y$ whenever $y \rhd x = x$ for $x, y \in X$. A \emph{subrack} of a rack $X$ is a non-empty subset $Y\subseteq X$ such that $(Y, \rhd)$ is also a rack.
\\[1pt] ${}$ \; Let $X$ be a rack. From (R1) and (R2) it follows that each $\varphi_x$ is a rack automorphism on $X$ which is called \emph{inner automorphism} of $X$. The \emph{inner group} $Inn(X)$ of a rack $X$ is generated by the inner automorphism $\varphi_x$, where $x\in X$. That is $Inn(X)=\left <\varphi_x |x\in X\right>$. A rack $X$ is called \emph{indecomposable or connected} if $Inn(X)$ acts transitively on $X$. Otherwise, $X$ is called \emph{decomposable}.

\end{defn}

\begin{exa}
A group $G$ is a quandle with $x \rhd y = xyx^{-1}$ for all $x, y \in G$. This quandle is called \emph{conjugation quandle}. Note that any subset $X\subseteq G$ stable under conjugation by $G$ (e.g. a union of conjugacy classes) is a crossed set with $x \rhd y = xyx^{-1}$ for all $x, y \in X$.
\end{exa}

\begin{exa}
Let $(G, +)$ be an abelian group and $Aut(G)$ is the automorphism group of $G$ with identity automorphism $1$. Then for any $\alpha \in Aut(G)$ we have a quandle structure on $G$ given by
\begin{center}
$x \rhd y = (1-\alpha)(x) + \alpha(y)$,
\end{center}
for all $x,y \in G$. This quandle is called the \emph{affine quandle (or Alexander quandle)} associated to the pair ($G, \alpha)$ and is denoted by \emph{Aff}$(G, \alpha)$. Note that for $\alpha = 1$  the affine quandle \emph{Aff}$(G, \alpha)$ is a \emph{trivial quandle} on $G$ with $x \rhd y=y$. For $\alpha = -1$  the affine quandle \emph{Aff}$(G, \alpha)$ is called \emph{kei or Takasaki quandle} with $x \rhd y = 2x-y$. In particular, for $\alpha = -1$ and $G=\mathbb{Z}_n=\{0,1,...,n-1\}$, the affine quandle \emph{Aff}$(G, \alpha)$ is called \emph{dihedral quandle}. Moreover, when $G$ is a finite field $\mathbb{F}_q$, where $q$ is a power of a prime number $p$, then $\alpha$ is defined by any non-zero element of $\mathbb{F}_q$. In this case the affine quandle \emph{Aff}$(G, \alpha)$ is written as \emph{Aff}$(\mathbb{F}_q, \alpha)$ or simply as \emph{Aff}$(q, \alpha)$.

\end{exa}

\begin{exa}
The group $\mathbb{Z}_n$ with $x \rhd y = y+1 \pmod n$ for all $x, y \in \mathbb{Z}_n$  is a rack which is not a quandle. This is a rack of cyclic type with no fixed point.
\end{exa}
Now we recall the definitions about the cycle structure of a finite rack (resp. quandle) from \cite{10}.

\begin{defn}
Let $X$ be a finite rack and $\varphi_x : X \longrightarrow X$ is an inner automorphism of $X$ for $x\in X$. Let $\varphi_x=\sigma_1\sigma_2...\sigma_k$ be the decomposition of $\varphi_x$ into product of disjoint cycles $\sigma_i$ for $1\leq i\leq k$. Let $l_i$ is the cycle length of $\sigma_i$. Then the list (a set with possible repeats) of all $l_i$ is called the \emph{pattern} of $\varphi_x$. The \emph{profile} of $X$ is the sequence of patterns of all $\varphi_x$.
\end{defn}
Note that if a rack $X$ is indecomposable then for all $x, y\in X$ there exists $\varphi\in Inn(X)$ such that $\varphi(x) =y$. Now for all $x, y\in X$ we have
\begin{center}
$(\varphi\circ\varphi_x)(y) = \varphi(x\rhd y)= \varphi(x) \rhd \varphi(y)=y \rhd \varphi(y)= (\phi_y\circ\varphi)(y)$.
\end{center}
 This implies $\phi_y  = \varphi\circ\varphi_x\circ\varphi^{-1}$ and $\varphi_x  = \varphi^{-1}\circ\varphi_x\circ\varphi$, where $\circ$ denotes the composition operation of $Inn(X)$. That is, any two inner automorphisms $\varphi_x, \phi_y$ of an indecomposable rack $X$ are mutually conjugate. Now since the conjugate permutations have same cycle structure, any two inner automorphisms $\varphi_x, \phi_y$ of a finite indecomposable rack $X$ have the same pattern, and therefore, the profile of a finite indecomposable rack is a constant sequence. In this way we can consider the pattern of any inner automorphism of a finite indecomposable rack as the profile of that rack. We write the profile of a finite indecomposable rack $X$ as follows:
\begin{center}
$1^{m_0}l_1^{m_1}l_2^{m_2}...l_k^{m_k}$,
\end{center}
where $l_i$ are the cycle lengths of $\sigma_i$ with $1<l_1<l_2<...<l_k$, and $m_0, m_1,..., m_k$ are the multiplicities of $1, l_1, l_2, ..., l_k$, respectively.
\\[1pt] ${}$ \; Note also that the order of $\varphi_x=\sigma_1\sigma_2...\sigma_k$, denoted by $ord(\varphi_x)$, is equal to the least common multiple ($lcm$) of the cycle lengths $l_i$ of disjoint cycles $\sigma_i$ of $\varphi_x$ for $1\leq i\leq k$. That is $ord(\varphi_x):=lcm(l_1, l_2, ..., l_k)$. The \emph{degree} of a finite indecomposable rack $X$ is the order of $\varphi_ x$, written $ord (\varphi_x)$, for some (equivalently, all) $x\in X $. The \emph{support} of $\varphi_x \in Inn(X)$, written $supp(\varphi_x)$, is the number of moved points by $\varphi_x$.

\begin{center}
	\section{\textsc{Obstructions on the Cycle Structure of Indecomposable Racks and Quandles}} \label{3.}
\end{center}
\paragraph{} In this section we provide obstructions on the profiles of finite indecomposable racks and crossed sets. We begin with the following definition.

\begin{defn}
Let $X$ be a finite indecomposable rack. For any subset $Y$ of $X$, the subrack of $X$ generated by $Y$
is the smallest subrack of $X$ containing $Y$.
\end{defn}

\begin{lem} \label{3.2.}
Let $X$ be a finite indecomposable rack. Let $Y$ be a subrack of $X$ with $Y\neq X$ and $Y^c=X \setminus Y$. Then $X$ is generated by $Y^c$.
\end{lem}

\begin{proof}
 Since $Y$ is a subrack of $X$, we conclude that $Y\rhd Y^c=Y^c$. Let
  \begin{center}
  $Z=\{y_1\rhd(y_2\rhd ...\rhd(y_{n-1}\rhd y_n) )\mid n\geq 1, y_1,...,y_n \in Y^{c}\}$.
  \end{center}
Then $y\rhd z\in Z$ for all $y \in Y^{c}, z\in Z$ by definition, and $y\rhd z\in Z$ for all $y \in Y$ by the self-distributivity of $\rhd$ and the $Y$-invariance of $Y^{c}$. Hence $Z$ is a non-empty $X$-invariant subset of $X$, and therefore equal to $X$ since $X$ is indecomposable.
\end{proof}

\begin{lem} \label{3.3.}
Let $X$ be a finite indecomposable rack such that $X=Y\cup Z$, for two subracks $Y$ and $Z$ of $X$. Then $X=Y$ or $X=Z$.
\end{lem}

\begin{proof}
Assume that $X\neq Y$. Then by Lemma \ref{3.2.} $X$ is generated by $Y^{c}\subseteq Z$. Since $Z$ is a subrack of $X$, one concludes that $X=Z$.
\end{proof}

 \begin{cor} \label{3.4.}
 Let $p, q \in \mathbb{Z}$ with $p,q \geq  2$. Let $X$ be an indecomposable rack and $x\in X$,
   \begin{center}
   $Y= \{y\in X \mid \varphi^p_x(y)=y \}, Z= \{z\in X \mid \varphi^q_x(z)=z \}$.
   \end{center}
 Assume that $X=Y\cup Z$. Then $X=Y$ or $X=Z$.
  \end{cor}

\begin{proof}  By the self-distributivity of $\rhd$, the sets $Y$ and $Z$ are subracks of $X$. Then the claim follows from Lemma \ref{3.3.}.
\end{proof}

\paragraph{} Corollary \ref{3.4.} is useful to provide the following obstruction on the profile of a finite indecomposable rack.

\begin{prop} \label{3.5.}
There is no finite indecomposable rack $X$ of profile $1^{m_0}l_1^{m_1}l_2^{m_2}...l_k^{m_k}$ such that $lcm(l_1,l_2,...,l_i)$ and $lcm(l_{i+1},l_{i+2},...,l_k)$ do not divide each other for $1\leq i<k$.
\end{prop}

\begin{proof} Suppose there exists a finite indecomposable rack $X$ with given profile. Let $p=lcm(l_1,l_2,...,l_i)$, and $q=lcm(l_{i+1},l_{i+2},...,l_k)$. Then $p,q\geq 2$ and $p,q$ do not divide each other. Let $x\in X$, and
 \begin{center}
 $Y= \{y\in X \mid \varphi^p_x(y)=y \}, Z= \{z\in X \mid \varphi^q_x(z)=z \}$.
 \end{center}
By the self-distributivity of $\rhd$, the sets $Y, Z$ are subracks of $X$ and $X=Y\cup Z$. However, by definition of $p$ and $q$, we have $X\neq Y$ and $X\neq Z$, a contradiction to Corollary \ref{3.4.}.
\end{proof}

\begin{rem}\label{3.6.}
By Proposition \ref{3.5.} there is no finite indecomposable rack $X$ of profile $1^{m_0}l_1^{m_1}l_2^{m_2}$ such that $l_1\nmid l_2$ (that is, $l_1$ does not divide $l_2$). Next consider the profile $1^{m_0}l_1^{m_1}l_2^{m_2}l_3^{m_3}$. Then we have two cases to consider, namely, when $l_k \nmid lcm(l_{k+1}, l_{k+2})\pmod 3$, and $l_k \mid lcm(l_{k+1}, l_{k+2})\pmod 3$ for $k\in \{1, 2, 3\}$. The case when $l_k \nmid lcm(l_{k+1}, l_{k+2})\pmod 3$ is excluded by Proposition \ref{3.5.}. However, the case when $l_k \mid lcm(l_{k+1}, l_{k+2})\pmod 3$ can not be excluded by Proposition \ref{3.5.}. For example the profile $1^{m_0}l_1l_2l_3$, with $(l_1, l_2, l_3)=(pq,pr,qr)$ for pairwise distinct primes $p,q,r$, can not be excluded by Proposition \ref{3.5.}. In the next section we consider such cases in details.

\end{rem}

\subsection{\text{Obstruction on the the Profile of Finite Indecomposable Crossed Sets}}
\paragraph{}  Let $X$ be a finite indecomposable crossed set with profile $1^{m_0}l_1^{m_1}l_2^{m_2}l_3^{m_3}$. Let $1< l_1< l_2 < l_3$, $l_1 \nmid l_2, \; l_1, l_2 \nmid l_3$ and $l_k \mid lcm(l_{k+1}, l_{k+2})\pmod 3$ for $k\in \{1, 2, 3\}$. Let $m_0, m_1,m_2,m_3 \in \mathbb{N}$ and let $p_1,p_2,...,p_r$ be pairwise distinct primes for positive integer $r$ such that
\begin{center}
$l_1=\prod\limits_{i=1}^{r} p_{i}^{a_i}$, $l_2=\prod\limits_{i=1}^{r} p_{i}^{b_i}$ and $l_3=\prod\limits_{i=1}^{r} p_{i}^{c_i}$,
\end{center}
where $a_i, b_i$ and $c_i$ are non-negative integers for all $1\leq i\leq r$. Let $A=\{ p_i| c_i=b_i>a_i\}$, $B=\{ p_i| a_i=c_i>b_i\}$, $C=\{ p_i| a_i=b_i>c_i\}$ and $D=\{ p_i| a_i=b_i=c_i\}$, then $\{p_1, p_2, ..., p_r\}=A\cup B\cup C\cup D$. Also let
 \begin{center}
 $p:=\prod\limits_{p_i\in C}p_i^{a_i}=\prod\limits_{p_i\in C}p_i^{b_i}$, \\$q:=\prod\limits_{p_i\in B}p_i^{a_i}=\prod\limits_{p_i\in B}p_i^{c_i}$, \\$r:=\prod\limits_{p_i\in A}p_i^{b_i}=\prod\limits_{p_i\in A}p_i^{c_i}$, \\$s:=\prod\limits_{p_i\in D}p_i^{a_i}=\prod\limits_{p_i\in D}p_i^{b_i}=\prod\limits_{p_i\in D}p_i^{c_i}$
 \end{center}
Then $p,q,r>1$ and $p,q,r,s$ are pairwise coprime integers. Let
 \begin{center}
   $p^\prime:=\prod\limits_{p_i\in C}p_i^{c_i}$, $q^\prime:=\prod\limits_{p_i\in B}p_i^{b_i}$, $r^\prime:=\prod\limits_{p_i\in A}p_i^{a_i}$,
 \end{center}
are such that , $p^\prime \mid p, \; q^\prime \mid q$ and $r^\prime \mid r$. Then we have
\begin{center}
  $l_1= \prod\limits_{p_i\in C}p_i^{a_i}\prod\limits_{p_i\in B}p_i^{a_i}\prod\limits_{p_i\in A}p_i^{a_i}\prod\limits_{p_i\in D}p_i^{a_i}=p qr^\prime s$,
\\$l_2=\prod\limits_{p_i\in C}p_i^{b_i}\prod\limits_{p_i\in B}p_i^{b_i}\prod\limits_{p_i\in A}p_i^{b_i}\prod\limits_{p_i\in D}p_i^{b_i}=p q^\prime rs$,
\\$l_3= \prod\limits_{p_i\in C}p_i^{c_i}\prod\limits_{p_i\in B}p_i^{c_i}\prod\limits_{p_i\in A}p_i^{c_i}\prod\limits_{p_i\in D}p_i^{c_i}=p^\prime qrs$.
\end{center}

\paragraph{} Note that if $C=\emptyset$, then $p=1=p^\prime$ and $l_1=qr^\prime s, l_2=q^\prime rs, l_3=qrs$. Since $r^\prime \mid r$ and $q^\prime \mid q$, $l_1 \mid l_3$ and $l_2 \mid l_3$. Therefore we assume that $A\neq\emptyset$, $B\neq\emptyset$ and $C\neq\emptyset$. Note also that if $s\neq1$ then the number of moved points is equal to \{the number of moved points for $s=1$\}. $s$. If we fix $p,q,r$, then there are finitely many choices for $p^\prime, q^\prime, r^\prime, s$. We consider the case when $(p^\prime, q^\prime, r^\prime)=(1,1,1)$. In this case we have $(l_1,l_2,l_3)=(pqs,prs,qrs)$. For example, for $(p, q, r)=(2,3,5)$ and $(p^\prime, q^\prime, r^\prime,s )=(1,1,1,1)$ we have $(l_1,l_2,l_3)=(6,10,15)$, and for $(p^\prime, q^\prime, r^\prime)=(1,1,1)$ and $(p, q, r, s)=(4,3,5,2)$ we have $(l_1,l_2,l_3)=(12,15,20)$.
\\[1pt] ${}$ \; Since $lcm(l_1, l_2)=\prod\limits_{i=1}^{r} p_{i}^{max(a_i, b_i)}$, therefore for all $1\leq i\leq r$, $l_3 \mid lcm(l_1, l_2)  \Leftrightarrow c_i\leq \max{\{a_i, b_i\}}$. Note that if there exist some $i$ such that $a_i\leq b_i< c_i$, then $c_i\nleq \max{\{a_i, b_i\}}$. This implies that $l_3 \nmid lcm(l_1, l_2)$ and there exists no such $X$ by Proposition \ref{3.5.}. Therefore we assume that there does not exist $i$ such that $a_i\leq b_i< c_i$. We show that these and additional assumptions lead to a contradiction when $l_3 \mid lcm(l_1, l_2)$ and $m_1=m_2=m_3=1$. The proof will be similar when $l_1 \mid lcm(l_2, l_3)$ or $l_2 \mid lcm(l_1, l_3)$ and $m_1=m_2=m_3=1$.
\\[1pt] ${}$ \; Assume that $X$ is a finite indecomposable crossed set with profile $1^{m_0}l_1l_2l_3$, where $l_1,l_2,l_3$ (defined as above) are such that $l_k \mid lcm(l_{k+1}, l_{k+2})\pmod 3$ for positive integer $k$. Let $x\in X$ and $t\geq 1$ be an integer. Then we define the set
\begin{center}
$X_t:= \{y\in X \mid \varphi^t_x(y)=y \}$.
\end{center}
Let $X^\prime_t=X_t\setminus X_1$ for all $t>1$. Then $X$ is the disjoint union of
non-empty sets $X_1, X^\prime_{pqs},X^\prime_{prs},X^\prime_{qrs}$. Now, for the obstruction on the profile $1^{m_0}l_1l_2l_3$ (defined as above) of a finite indecomposable crossed set $X$, we prove the following lemmas.

\begin{lem}\label{3.7.}
The set $X_t$ is a subrack of $X$ for all $t\geq 1$. In particular, $y\rhd (X\setminus X_t)=X\setminus X_t$ for all $y \in X_t$.
\end{lem}

\begin{proof}
Since $X$ is a crossed set, $x\in X_t$ and therefore $X_t \neq \emptyset$. Let $y_1,y_2 \in X_t$. Then $\varphi^t_1(y_1\rhd y_2)=\varphi^t_1(y_1)\rhd \varphi^t_1(y_2)=y_1\rhd y_2$. This implies that $X_t$ is a subrack of $X$. For all $y \in X_t$,  $y\rhd (X\setminus X_t)=y\rhd X\setminus y\rhd X_t=\varphi_y(X)\setminus \varphi_y(X_t)=X\setminus X_t$, since $\varphi_y$ is a bijection on $X$ and also on $X_t$ for all $y \in X_t$.
\end{proof}

\begin{lem} \label{3.8.}
For all $y\in X^\prime_{pqs}$ there exists  $z\in X^\prime_{prs}$ such that $y\rhd z \neq z$.
\end{lem}

\begin{proof}
Suppose $y\rhd z = z$ for all $z\in X^\prime_{prs}$, then $y^\prime\rhd z = z$ and $z\rhd y^\prime = y^\prime$ for all $z\in X^\prime_{prs}$ and $y^\prime \in X^\prime_{pqs}$, since $X$ is a crossed set and $\varphi_x$ acts transitively on 
$X^\prime_{pqs}$. Then $Y=X_{pqs}\cup X_{prs}$ is a subrack of $X$ with $Y \neq X$ and $Y\cup X_{qrs}=X$. This is a contradiction to Lemma \ref{3.3.}. Hence $y\rhd z \neq z$.
\end{proof}

\begin{lem} \label{3.9.}
Let $y\in X^\prime_{pqs}$ and $z \in X\setminus X_{pqs}$ with $y\rhd z\neq z$. Let $t$ be the smallest positive integer with $\varphi^t_y(z)=z$. Then $t=prs$ or $t=qrs$.
\end{lem}

\begin{proof}
Since $X$ is indecomposable, $\varphi_y$ is the product of a $pqs-$, a $prs-$, and a $qrs-$ cycle. If $y\rhd z^\prime\neq z^\prime$ for some $z^\prime \in X\setminus X_{pqs}$, then all entries of the cycle of $\varphi_y$ containing $z^\prime$ are in $X\setminus X_{pqs}$, since $X\setminus X_{pqs}$ is invariant under $X_{pqs}$ by Lemma \ref{3.7.}.
\\[1pt] ${}$ \; Since $y \in X_{pqs}$, and $X_{pqs}$ is a subrack of $X$ therefore we have $\varphi_x^{pqs}\varphi_y=\varphi_y\varphi_x^{pqs}$. Hence, $z \in supp(\sigma_i)$ with $\varphi_y^t(z)= z$ imply that $y\rhd u \neq u$ and $\varphi_y^t(u)= u$ for $u=\varphi_x^{pqs}(z)$. 
Since all $\varphi_x^{pqs}$-orbits of $X^\prime_{prs}$ and of $X^\prime_{pqs}$ have $rs$ elements,
the number of elements of $X\setminus X_{pqs}$ moved by $\varphi_y$ is a multiple of $rs$. Therefore the elements of $X\setminus X_{pqs}$ moved by $\varphi_y$ are contained in the $prs$- and the $qrs$-cycles of $\varphi_y$.
\end{proof}

\begin{lem} \label{3.10.}
Let $y\in X^\prime_{pqs}$. Then there exist $z \in X^\prime_{prs},f \in X^\prime_{qrs}$ such that $ z\rhd f=y$ or $f\rhd z=y$.
\end{lem}

\begin{proof}
Let $Y= X_{prs}\cup X_{qrs}$. If $Y$ is a subrack of $X$, then $X=Y\cup X_{pqs}$, which is a contradiction to Lemma \ref{3.3.}. Hence there exist $z^\prime, f^\prime \in Y$ such that $y^\prime :=z^\prime\rhd f^\prime \in X\setminus Y=X^\prime_{pqs}$. If $z^\prime, f^\prime$ are both in $X_{prs}$ or $X_{qrs}$, (in particular if $z^\prime \in X_1$ or $f^\prime \in X_1$) then $z^\prime\rhd f^\prime \in Y$ since $X_{prs}$ and $X_{qrs}$ are subracks of $X$. Thus $(z^\prime, f^\prime)\in X^\prime_{prs}\times X^\prime_{qrs}\cup X^\prime_{qrs}\times X^\prime_{prs}$. Then the claim of the lemma follows from the fact that $\varphi_x$ acts transitively on $X^\prime_{pqs}$ and permutes both $X^\prime_{prs}$ and $X^\prime_{qrs}$.
\end{proof}

\begin{lem} \label{3.11.}
Assume that $X^\prime_{prs}$ and $X^\prime_{qrs}$ are subracks of $X$. Then $X^\prime_{pqs}$ is not a subrack of $X$.
\end{lem}

\begin{proof}
 Assume that $X^\prime_{pqs}$ is a subrack of $X$. Let $Y=X^\prime_{pqs}\cup X^\prime_{prs}\cup X^\prime_{qrs}$. Lemma \ref{3.7.} implies that $X^\prime_{pqs}$ permutes $X\setminus X_{pqs}=X^\prime_{prs}\cup X^\prime_{qrs}$ and hence it permutes $Y$, since $X^\prime_{pqs}$ also permutes itself being a subrack of $X$. Similarly, $X^\prime_{prs}$ and $X^\prime_{qrs}$ permute, respectively, $X^\prime_{prs}$, $X\setminus X_{prs}=X^\prime_{pqs}\cup X^\prime_{qrs}$ and $X^\prime_{qrs}$, $X\setminus X_{qrs}=X^\prime_{pqs}\cup X^\prime_{prs}$. Hence $Y$ is a subrack of $X$. This is a contradiction to the fact that $X=Y\cup X_1, Y \neq X, X_1 \neq X$, and to Lemma \ref{3.3.}.
\end{proof}

\begin{lem} \label{3.12.}
Let $y\in X^\prime_{pqs}$. Then there exist $z \in X^\prime_{prs}$ and $f \in X^\prime_{qrs}$ with $y\rhd z\in X^\prime_{qrs}$, $y\rhd f\in X^\prime_{prs}$.
\end{lem}

\begin{proof}
Assume that $y\rhd z\in X^\prime_{prs}$ for all $z\in X^\prime_{prs}$. Then $y^\prime\rhd z\in X^\prime_{prs}$ for all $z \in X^\prime_{prs}$ and $y\rhd X^\prime_{qrs}=X^\prime_{qrs}$. By Lemma \ref{3.8.} $y\rhd z \neq z$ and $y\rhd f \neq f$. This implies that the restrictions $\varphi_y|_{X^\prime_{prs}}$ and $\varphi_y|_{X^\prime_{qrs}}$ are not the identity. Therefore $\varphi_y$ has at least two cycles consisting of elements of $X^\prime_{prs}\cup X^\prime_{qrs}$. By Lemma \ref{3.9.} these are the $prs-$ and $qrs-$ cycles of $\varphi_y$. Therefore the $prs-$ and $qrs-$ cycles of $\varphi_y$ consist of the elements of $X^\prime_{prs}$ and the elements of $X^\prime_{qrs}$, respectively.
\\[1pt] ${}$ \; If $z\rhd y^\prime \in X^\prime_{pqs}$ for all $y^\prime \in X^\prime_{pqs}$, $z\in X^\prime_{prs}$, then $Y= X_{pqs}\cup X_{prs}$ is a subrack of $X$, a contradiction to $X= Y\cup X_{qrs}$ and Lemma \ref{3.3.}. Thus there exists $z\in X^\prime_{prs}$ with $z\rhd y \notin X^\prime_{pqs}$. Then $z\rhd y \in X^\prime_{qr s}$. Now $z$ is in $prs-$ cycle of $\varphi_y$. Thus the $prs-$cycle of $\varphi_{z\rhd y}$ consists of the elements of $z\rhd X^\prime_{prs}\subseteq X_{prs}$, and one of these elements is $z\in X^\prime_{prs}$. Since $z\rhd y\in X^\prime_{qrs}$, the elements of the $prs-$cycle of $\varphi_{z\rhd y}$ also belong to $X \setminus X_{qrs}$, and hence to $X^\prime_{prs}$. Thus $z$ permutes $X^\prime_{prs}$ and hence $X^\prime_{prs}$ is a subrack of $X$. By the same reason, $X^\prime_{qrs}$ is a subrack of $X$.
\\[1pt] ${}$ \; Consider now the $qrs-$ cycle of $\varphi_{z\rhd y}$, which consists of the elements of $z\rhd X^\prime_{qrs} \subseteq X^\prime_{pqs} \cup X^\prime_{qrs}$. Since $z\rhd y\in X^\prime_{qrs}$ permutes $z\rhd X^\prime_{qrs}$ and it maps $X^\prime_{qrs}$ to $X_{qrs}$ and $X^\prime_{pqs}$ to $X^\prime_{pqs} \cup X^\prime_{prs}$, respectively, we conclude that $z\rhd X^\prime_{qrs}\subseteq X^\prime_{qrs}$ or $z\rhd X^\prime_{qrs}\subseteq X^\prime_{pqs}$. If $z\rhd X^\prime_{qrs}\subseteq X^\prime_{qrs}$, then the first part of the proof applied to $z$ instead of $y$ implies that $X^\prime_{pqs}$ is a subrack of $X$. This is a contradiction to Lemma \ref{3.11.}. Assume now that $z\rhd X^\prime_{qrs}\subseteq X^\prime_{pqs}$. Since $\varphi_x$ acts transitively on $X^\prime_{prs}$ and on $X^\prime_{pqs}$, we conclude that

\begin{center}
  $X^\prime_{prs} \rhd X^\prime_{qrs}=X^\prime_{pqs}$.
\end{center}
By applying $X^\prime_{pqs}$ to this equation and using the self-distributivity of $\rhd$ one concludes that

\begin{center}
$X^\prime_{pqs} \rhd X^\prime_{pqs}=X^\prime_{pqs} \rhd(X^\prime_{prs}\rhd X^\prime_{qrs})\subseteq (X^\prime_{pqs} \rhd X^\prime_{prs})\rhd(X^\prime_{pqs} \rhd X^\prime_{qrs})=(X^\prime_{prs} \rhd X^\prime_{qrs})=X^\prime_{pqs}$.
\end{center}
Again we conclude that $X^\prime_{pqs}$ is a subrack of $X$ which is a contradiction to Lemma \ref{3.11.}.
\end{proof}

\begin{lem} \label{3.13.}
Let $y \in X^\prime_{pqs}$. Then $y\rhd x\in X^\prime_{pqs}$.
\end{lem}

\begin{proof}
Since $y, x \in X_{pqs}$, we conclude that $y\rhd x \in X_{pqs}$. Assume that $y \rhd x\notin X^\prime_{pqs}$. Then $y \rhd x\in X_1$. The $prs-$ and $qrs-$cycles of $\varphi_{y\rhd x}$ consist of the elements of $y\rhd X^\prime_{prs}$ and $y\rhd X^\prime_{qrs}$, respectively. Since $y \in X^\prime_{pqs}$,
the ${p qs}-$ and $pr s-$cycles of $\varphi_{y\rhd x}$ contain together all elements of $X^\prime_{prs} \cup X^\prime_{qrs}$, by Lemma 2.1. Moreover, by conjugating with $\varphi_x$ we conclude that if an element of $X^\prime_{prs}$ (or $X^\prime_{qrs}$, respectively) is contained in a $qrs-$cycle (in a $prs-$cycle, respectively), then all the elements of $X^\prime_{prs}$ (of $X^\prime_{qrs}$, respectively) do
so. Since $prs\neq qrs$, this is not possible. Hence $\varphi_y$ permutes both $X^\prime_{prs}$ and $X^\prime_{qrs}$. This is a contradiction to Lemma \ref{3.12.}.
\end{proof}

\begin{lem} \label{3.14.}
Let $y\in X^\prime_{pqs}$. Let $\varphi_y=\sigma_1\sigma_2\sigma_3$ be the decomposition of $\varphi_y$ into the product of a $pqs-$, a $prs-$, and a $qrs-$ cycle. Then the entries of $\sigma_1$ belong to $X_{pqs}$ and the entries of $\sigma_2$ and $\sigma_3$ belong to $X^\prime_{prs}\cup X^\prime_{qrs}$.
\end{lem}

\begin{proof}
By Lemma \ref{3.10.} there exist $z \in X^\prime_{prs}$ and $f \in X^\prime_{qrs}$ such that $ z\rhd f=y$ or $f\rhd z=y$. By interchanging $p$ and $q$ if necessary, we may assume that $ z\rhd f=y$.
\\[1pt] ${}$ \; Note that $f\rhd x \neq x$ and $f\rhd z\neq z$ since $X$ is a crossed set. Moreover,
since $f,x \in X_{qrs}$ and $X_{qrs}$ is a subrack of $X$, and since $z \notin X_{qrs}$, the elements $z$ and $x$ are in different cycles of  $\varphi_f$. Let $\varphi_f=(x...)(z...)(...)$ be the product of disjoint cycles. Then
\begin{center}
$\varphi_y=\varphi_{z\rhd f}=\varphi_z\varphi_f\varphi^{-1}_z=(z \rhd x...)(z...)(...).$
\end{center}
Since $z\rhd x \in X^\prime_{prs}$ by Lemma \ref{3.13.} and $z \in X^\prime_{prs}$, $\varphi_y$ has two cycles containing elements of $X^\prime_{prs}\cup X^\prime_{qrs}$. These are the $prs-$ and $qrs-$cycles by Lemma \ref{3.9.}. This implies the claim.
\end{proof}

\begin{prop} \label{3.15.}
There is no finite indecomposable crossed set $X$ with profile $1^{m_0}l_1l_2l_3$, where $1< l_1< l_2 < l_3$, $l_1 \nmid l_2, \; l_1, l_2 \nmid l_3$ and $l_k \mid lcm(l_{k+1}, l_{k+2})\pmod 3$ for $k\in \{1, 2, 3\}$.
\end{prop}

\begin{proof}
Assume to the contrary that $X$ is a finite indecomposable crossed set with profile $1^{m_0}l_1l_2l_3$ where $1< l_1< l_2 < l_3$, $l_1 \nmid l_2, \; l_1, l_2 \nmid l_3$ and $l_k \mid lcm(l_{k+1}, l_{k+2})\pmod 3$ for $k\in \{1, 2, 3\}$. Let $l_1, l_2, l_3, X_t$ and $X^\prime_t$ be defined as above for integer $t\geq 1$.
\\[1pt] ${}$ \; Let $y \in X^\prime_{pqs}$ and $z \in X^\prime_{prs}$. Then $z\rhd x\neq x$ and $\varphi^{prs}_z(x)=x$ by Lemma \ref{3.13.}. Therefore $\varphi^{prs}_{y\rhd z}(y\rhd x)=y\rhd x$. Moreover, $y\rhd x\in X^\prime_{pqs}$ by Lemma \ref{3.13.}. Lemma \ref{3.7.} implies that $y\rhd z\in X^\prime_{prs}\cup X^\prime_{qrs}$. If $y\rhd z\in X^\prime_{prs}$, then the entries of $prs-$cycle of $\varphi_{y\rhd z}$ belong to $X_{prs}$ by Lemma \ref{3.14.}, in contradiction to $y\rhd x\in X^\prime_{pqs}$. Thus $y\rhd z\in X^\prime_{qrs}$. This implies that $y\rhd X^\prime_{prs} \subseteq X^\prime_{qrs}$ and $y\rhd X^\prime_{qrs} \subseteq X^\prime_{prs}$ by symmetry. This is impossible since $\mid X^\prime_{prs}\mid= prs\neq qrs = \mid X^\prime_{qrs}\mid$.
\end{proof}

\begin{center}
	\section{\textsc{Hayashi's Conjecture on Finite Indecomposable Racks and Quanldes}} \label{4.}
\end{center}
\paragraph{} In \cite{7} C. Hayashi proposed a conjecture on the cycle structure of finite connected quandles which we recall here for a finite indecomposable rack.

\begin{conj}
Let $X$ be a finite indecomposable rack with profile $1^{m_0}l_1^{m_1}l_2^{m_2}...l_k^{m_k}$. Then $l_i|l_k$ (i.e., $l_i$ divides $l_k$) for any integer $i$ with $1\leq i\leq k-1$.
\end{conj}
Based on the obstructions on the profiles of finite indecomposable racks and crossed sets provided in Section \ref{3.}, we verify some cases of the Hayashi's conjecture on finite indecomposable racks.
\\\\\emph{Case 1}. The Hayashi's conjecture is trivially true for a finite indecomposable rack with profile $1^{m_0}\ell_1^{m_1}$. The examples of finite indecomposable rack and quandles with profile $1^{m_0}\ell_1^{m_1}$ are the connected quandles of cycle type (see \cite{12}) and the dihedral quandles on $\mathbb{Z}_n$ when $n$ is odd. For indecomposable dihedral quandle on $\mathbb{Z}_n$ we have $\varphi_x(y)=2x-y\pmod n$ and
\begin{center}
$\varphi_x =\prod\limits_{y=1}^{[\frac{n-1}{2}]}(x+y\;\;x-y)\pmod n$.
\end{center}
Therefore, the profile of an indecomposable dihedral quandle on $\mathbb{Z}_n$ is $1^{m_0}l_1^{m_1}$ with $l_1=2$.
\\\\\emph{Case 2}. The Hayashi's conjecture is true for a finite indecomposable rack with profile $1^{m_0}l_1^{m_1}l_2^{m_2}$, since by Proposition \ref{3.5.} there is no finite indecomposable rack $X$ with profile $1^{m_0}l_1^{m_1}l_2^{m_2}$ such that $l_1$ does not divide $l_2$.
\\\\\emph{Case 3}. The Hayashi's conjecture is true for a finite indecomposable crossed set with profile $1^{m_0}l_1l_2l_3$, since by Proposition \ref{3.5.} and Proposition \ref{3.15.} there is no finite crossed set with profile profile $1^{m_0}l_1l_2l_3$ such that $l_1$ or $l_2$ or both $l_1$ and $l_2$ do not divide $l_3$. One such case with smallest $l_1, l_2, l_3$ is when $(l_1, l_2, l_3)=(6, 10, 15)$, by which it follows that the Hayashi's conjecture is true for a finite indecomposable crossed set $X$ with $\varphi_x \in Inn(X)$ such that $supp(\varphi_x)\leq 31$.
\\\\\emph{Case 4}. The Hayashi's conjecture is true for a finite indecomposable rack $X$ such $X$ is a quandle and at least one of the equations $x \rhd y = y, x \rhd (y \rhd x) = y$ holds for all $x, y \in X$. Such racks are called braided. The finite indecomposable braided racks have been studied and classified in \cite{8}. Let $X$ be a finite indecomposable rack. Then the degree of $X$, written $deg(X):=ord(\varphi_x)$, is $1, 2, 3, 4$ or $6$ (see \cite{8} Proposition 6). When $deg(X)=ord(\varphi_x) = 1$ then $\varphi_x(y) = x\rhd y = y$ and therefore the profile of $X$ is $1^{m_0}$. When $deg(X) =ord(\varphi_x)=2$, the profile of $X$ is $1^{m_0}l_1^{m_1}$ with $l_1 = 2$. Similarly, when $deg(X) =ord(\varphi_x)=3$, the profile of $X$ is $1^{m_0}l_1^{m_1}$ with $l_1 = 3$. Next consider the case when $deg(X) =ord(\varphi_x)=4$. Then the profile of $X$ is either $1^{m_0}l_1^{m_1}$ with $l_1 = 4$ or $1^{m_0}l_1^{m_1}l_2^{m_2}$ with $l_1 = 2$ and $l_2 = 4$ since $lcm(2, 4)= ord(\varphi_x)=4$. Finally, consider the case when $deg(X) =ord(\varphi_x)=6$. Then the profile of $X$ is either $1^{m_0}l_1^{m_1}$ with $l_1 = 6$ or $1^{m_0}l_1^{m_1}l_2^{m_2}$ with $l_1 = 2$ and $l_2 = 6$ or $1^{m_0}l_1^{m_1}l_2^{m_2}$ with $l_1 = 3$ and $l_2 = 6$ since $ord(\varphi_x)=lcm(2, 6)=lcm(2, 6)=6$. Note that by \emph{Case 2} the profile of a finite indecomposable braided rack $X$ cannot be $1^{m_0} l_1^{m_1} l_2^{m_2} $ with  $l_1 = 2$ and $l_2 = 3$.
\\\\\emph{Case 4}. The Hayashi's conjecture is true for a finite indecomposable affine quandle \emph{Aff}$(G, \alpha)$. In order to see this first note that for an affine quandle \emph{Aff}$(G, \alpha)$ we have
\begin{center}
$\varphi_x(y)= x \rhd y = (1-\alpha)(x) + \alpha(y)$
\end{center}
Now if we take $x$ as the identity $0$ of $G$ we get $\varphi_0(y)= \alpha(y)$ for all $y\in G$. This implies that $\varphi_0= \alpha$, and therefore, the cycle structure of $\varphi_0$ is equal to the cycle structure of $\alpha \in Aut(G)$. The cycle structure of automorphisms of a finite group $G$ has been studied by many, notably in  \cite{2}, \cite{6}. In these studies a cycle $\sigma$ of $\alpha$ is called a \emph{regular cycle or orbit} if the length of $\sigma$ is equal to $ord(\alpha)$ (which is the least common multiple of the cycle lengths of $\alpha$). Therefore if $\alpha$ has a regular cycle $\sigma$ then all cycle lengths of $\alpha$ divide the largest cycle length which is $ord(\alpha)$. By \cite{2} (Corollary $2.10$) any nilpotent group has a regular cycle. Now since the affine quandle \emph{Aff}$(G, \alpha)$ is defined on an abelian group $G$ which is nilpotent, the automorphism $\alpha$ of $G$ has a regular cycle. Therefore the Hayashi's conjecture is true for a finite indecomposable affine quandle \emph{Aff}$(G, \alpha)$.
\\\\\emph{Case 5}. The Hayashi's conjecture is true for all finite indecomposable racks of size $p$ and $p^2$, where $p$ is a prime number. By \cite{3}, \cite{4} we know that a finite indecomposable rack of size $p$ or $p^2$ is either affine or of cyclic type with no fixed point. The profile of a rack of cyclic type with no fixed point is $l_1$.
\\\\\emph{Case 6}. The Hayashi's conjecture is true for all finite indecomposable quandles of size at most $47$. The list of all indecomposable quandles of size $n\leq 47$ is available in a GAP package called \emph{Rig} (see \cite{5}) with notation $SmallQuandle(n, q(n))$, where $q(n):=$ quandle number of size $n$. The profiles of these small quandles can be computed with the following function
\begin{center}
$Profile:=CycleLengths(Permutations(q)[1],[1..n])$.
\end{center}
For example the profile of $SmallQuandle(42, 7)$ is $1^2.2^2.3^4.6^4$. By inspection we also observed that the profile of a finite indecomposable affine quandle of prime size $p$ with $p\leq 47$ is $1^{m_0}l_1^{m_1}$, and the profile of a finite indecomposable affine quandles of size $p^2$ with $p^2\leq 47$ is either $1^{m_0}l_1^{m_1}$ or $1^{m_0}l_1^{m_1}l_2^{m_2}$.

\paragraph{\textbf{Acknowledgement.}} The author is grateful to Istvan Heckenberger for introducing the problem and useful discussion. This work was supported by German Academic Exchange Service (DAAD).


\begin{center}

\end{center}
Naqeeb ur Rehman, Allama Iqbal Open University Islamabad, Pakistan.\\
Email: naqeeb@aiou.edu.pk
\end{document}